\documentclass[reqno]{amsart}
\usepackage{amssymb}
\usepackage{hyperref}
\usepackage{mathrsfs}

\AtBeginDocument{{\noindent\small
\sc{Journal of Mathematical Analysis}\newline
ISSN:  2217-3412, URL: http://www.ilirias.com\newline
Volume XX Issue XX(2015), Pages XX-XX.}

\thanks{\copyright 2010 Ilirias Publications, Prishtin\"e, Kosov\"e.}
\vspace{9mm}}

\begin{document}
\title[ Bivariate Chebyshev on Simplicial]
{Bivariate Chebyshev-I Weighted Orthogonal Polynomials on Simplicial Domains}

\author[M.A. AlQudah]
{Mohammad A. AlQudah}  

\address{Mohammad A. AlQudah \newline
Department of Mathematics, Northwood University, Midland MI 48640 USA}
\email{alqudahm@northwood.edu}

\thanks{Submitted January 6, 2015. Published January 2, 2008.}


\begin{abstract}
We construct a simple closed-form representation of degree-ordered system of
bivariate Chebyshev-I orthogonal polynomials $\mathscr{T}_{n,r}(u,v,w)$  on simplicial domains. 
We show that these polynomials $\mathscr{T}_{n,r}(u,v,w),$
$r=0,1,\dots,n;$ $n\geq 0$ form an orthogonal system with respect to the Chebyshev-I weight function.
\end{abstract}

\maketitle
\numberwithin{equation}{section}
\newtheorem{theorem}{Theorem}[section]
\newtheorem{lemma}[theorem]{Lemma}
\newtheorem{proposition}[theorem]{Proposition}
\newtheorem{corollary}[theorem]{Corollary}
\newtheorem*{remark}{Remark}
\newtheorem{definition}[theorem]{Definition}

\section{Introduction and Motivation}
Orthogonal polynomials have been studied thoroughly \cite{Olver,Szego}, the Chebyshev polynomials of the first kind (Chebyshev-I) are among these polynomials.
Although the main definitions and basic properties were considered many years ago \cite{Dunkl,Suetin}, the cases of bivariate
or more variables have been studied by few scholars. 

Chebyshev-I polynomials ($T_{n,r}(u,v,w)$) are defined in many textbooks \cite{Abramowitz,Olver}, they  are orthogonal to each polynomial
of degree $\leq n-1,$ with respect to the defined weight function. However, for $r\neq s,$ $T_{n,r}(u,v,w)$ and $T_{n,s}(u,v,w)$ are not orthogonal with respect to the weight function.

Farouki \cite{Farouki} defined orthogonal polynomials with respect to the weight function $\mathrm{W}(u,v,w)=1$ on a triangular domain.
These polynomials $P_{n,r}(u,v,w)$ are orthogonal to each polynomial of degree $\leq n-1,$ and also orthogonal to each polynomial $P_{n,s}(u,v,w),$ where $r\neq s.$

In this paper, we construct bivariate Chebyshev-I weighted orthogonal polynomials $\mathscr{T}_{n,r}(u,v,w)$ with respect to the weight function
$\mathrm{W}(u,v,w)=u^{-\frac{1}{2}}v^{-\frac{1}{2}}(1-w)^\gamma, \gamma\geq 1,$ on triangular domain $T.$ We show that these bivariate polynomials $\mathscr{T}_{n,r}(u,v,w),$ $r = 0,1,\dots ,n;$ $n\geq 0,$ 
form an orthogonal system over $T$ with respect to $\mathrm{W}(u,v,w)=u^{-\frac{1}{2}}v^{-\frac{1}{2}}(1- w)^\gamma, \gamma\geq 1.$

\subsection{Bernstein and Orthogonal polynomials}\label{Univar-Ultra-Poly}
Consider a triangle $T$ defined by its three vertices $\textbf{p}_{k}=(x_{k},y_{k}),$  $k=1,2,3.$
For each point \textbf{p} located inside the triangle, there is a sequence of three numbers $u,v,w\geq 0$ such that \textbf{p} 
can be written uniquely as a convex combination of the three vertices, $\textbf{p}=u\textbf{p}_{1}+v\textbf{p}_{2}+w\textbf{p}_{3},$
where $u+v+w=1.$
The three numbers $u=\frac{area(\textbf{p},\textbf{p}_{2},\textbf{p}_{3})}{area(\textbf{p}_{1},\textbf{p}_{2},\textbf{p}_{3})},$ $v=\frac{
area(\textbf{p}_{1},\textbf{p},\textbf{p}_{3})}{area(\textbf{p}_{1},\textbf{p}_{2},\textbf{p}_{3})},$ $w=\frac{area(\textbf{p}_{1},\textbf{p}_{2},\textbf{p})}{%
area(\textbf{p}_{1},\textbf{p}_{2},\textbf{p}_{3})}$ indicate the barycentric "area" coordinates of the point \textbf{p} with respect to the triangle.

Although there are three coordinates, there are only two degrees of freedom, since $u+v+w=1.$ Thus every point is uniquely defined by
any two of the barycentric coordinates. That is, the triangular domain defined as $$T=\{(u,v,w):u,v,w\geq 0, u+v+w=1\}.$$ 

\begin{definition}The $n+1$ Bernstein polynomials of degree $n$ are defined by 
\begin{equation}\label{bern}
B_{i}^{n}(u)=\binom {n}{i} u^{i}(1-u)^{n-i}, \hspace{.1in}\text{for}\hspace{.1in} i=0,1,\dots,n,
\end{equation}
where $\binom{n}{i}$ is the binomial coefficients. For $\zeta=(i,j,k)$ denote triples of non-negative integers such that $\left|\zeta\right|=i+j+k,$ then the generalized Bernstein polynomials of
degree $n$ are defined by the formula
$$B_{\zeta}^{n}(u,v,w)=\binom{n}{\zeta}u^{i}v^{j}w^{k},\hspace{.1in}\left|\zeta\right|=n,\hspace{.1in}\text{where}\hspace{.1in}\binom{n}{\zeta}=\frac{n!}{i!j!k!}.$$
\end{definition}
The generalized Bernstein polynomials have a number of useful  analytical and elegant geometric properties \cite{Farin}. In addition, the Bernstein basis polynomials of degree $n$ form a basis for the space $\Pi_{n}$ of polynomials of degree at most $n.$
Thus, with the revolt of computer graphics, Bernstein polynomials on $[0,1]$ became important in the form of B\'{e}zier curves,
and the polynomials determined in the Bernstein (B\'{e}zier) basis enjoy considerable popularity in computer-aided Geometric design (CAGD) applications.

Degree elevation is a common situation in these applications, where polynomials given in the basis of degree $n$ have to be represented in the basis of higher degree.
For any polynomial $p(u,v,w)$ of degree $n$ can be written using B\'{e}zier coefficients $d_{\zeta}$ in the Bernstein form
\begin{equation} \label{deg elevation} p(u,v,w)=\sum_{\left|\zeta\right|=n}d_{\zeta}B_{\zeta}^{n}(u,v,w).\end{equation}
 
With the use of degree elevation algorithm for the Bernstein representation, the polynomial $p(u,v,w)$ in \eqref{deg elevation} can be written
(multiplying both sides by $1=u+v+w$) as
$$p(u,v,w)=\sum_{\left|\zeta\right|=n+1}d_{\zeta}^{(1)}B_{\zeta}^{n+1}(u,v,w).$$
The new coefficients defined in \cite {Farin,Hoschek} as $d_{\zeta}^{(1)}=\frac{1}{n+1}(id_{i-1,j,k}+jd_{i,j-1,k}+kd_{i,j,k-1})$ where $\left|\zeta\right|=n+1.$ 
Moreover, the next integration is one of the interesting analytical properties of the Bernstein polynomials $B_{\zeta}^{n}(u,v,w);$
$$\iint_{T}B_{\zeta}^{n}(u,v,w)dA=\frac{\Delta}{\binom{n+2}{2}},$$
where $\Delta$ is the double the area of $T$ and $\binom{n+2}{2}$  is the dimension of Bernstein
polynomials over the triangle. This means that the Bernstein polynomials partition the
unity with equal integrals over the domain; in other words, they are equally weighted
as basis functions. 

\begin{definition}
Let $p(u,v,w)$ and $q(u,v,w)$ be two bivariate polynomials over $T,$ then we define their inner product over $T$ by
$$\left\langle p,q\right\rangle = \frac{1}{\Delta}\iint_{T}pq dA.$$ 
With the inner product defined, we say that the two polynomials  $p(u,v,w)$ and $q(u,v,w)$  are orthogonal if $\left\langle p,q\right\rangle= 0.$
\end{definition}

For $m \geq 1,$ let $\mathfrak{L_{m}}$ denote the space of polynomials of degree $m$ that are orthogonal to all polynomials of degree $\leq m$
over a triangular domain $T,$ i.e., 
$$\mathfrak{L_{m}}=\{p\in \Pi_{m}:p\perp \Pi_{m-1}\}.$$ 

For an integrable function $f (u,v,w)$ over $T,$ consider the operator
$S_{n}(f)$ defined in \cite{Derriennic} as
$$S_{n}(f)=(n+1)(n+2)\sum_{\left|\zeta\right|=n}\left\langle f,B_{\zeta}^{n} \right\rangle B_{\zeta}^{n}.$$
$\text{For}\hspace{.1in} n \geq m, \hspace{.1in} \lambda_{m,n}=\frac{(n+2)!n!}{(n+m+2)!(n-m)!}$
is an eigenvalue of $S_{n},$ and $\mathfrak{L_{m}}$ is the corresponding eigenspace.
The following lemmas will be used in the proof of the main results. 
\begin{lemma}\cite{Farouki}. Let $p=\sum_{\left|\zeta\right|=n}c_{\zeta}B_{\zeta}^{n}\in\mathfrak{L_{m}}$ and let $q=\sum_{\left|\zeta\right|=n}d_{\zeta}B_{\zeta}^{n}\in \Pi_{n}$ with $m\leq n.$ Then,
$$\left\langle p,q \right\rangle=\frac{(n!)^2}{(n+m+2)!(n-m)!}\sum_{\left|\zeta\right|=n}c_{\zeta}d_{\zeta}.$$
\end{lemma}

\begin{lemma}\label{iffthm} \cite{Farouki}. Let $p=\sum_{\left|\zeta\right|=n}c_{\zeta}B_{\zeta}^{n}\in\Pi_{n}.$ Then we have
\begin{equation}
p\in\mathfrak{L_{n}}\iff \sum_{\left|\zeta\right|=n}c_{\zeta}d_{\zeta}=0 \hspace{.1in} \forall q=\sum_{\left|\zeta\right|=n}d_{\zeta}B_{\zeta}^{n}\in \Pi_{n-1}.
\end{equation}
\end{lemma}

\subsection{Factorials}
We present some results concerning factorials, double factorials, and combinatorial identities.
The double factorial of an integer $n$ is given by
\begin{equation}\label{douvle-factrrial}
\begin{aligned}
(2n-1)!!&=(2n-1)(2n-3)(2n-5)\dots(3)(1) \hspace{.2in}  \text{if $n$ is odd} \\
n!!&=(n)(n-2)(n-4)\dots(4)(2) \hspace{.63in} \text{if $n$ is even},
\end{aligned}
\end{equation}
where $0!!=(-1)!!=1.$ From \eqref{douvle-factrrial}, we have the following definition. 
\begin{definition}For an integer $n,$ the double factorial is defined as
\begin{equation}\label{doublefac} n!!=\left\{\begin{array}{ll}2^{\frac{n}{2}}(\frac{n}{2})! & \text{if $n$ is even} \\\frac{n!}{2^{\frac{n-1}{2}}(\frac{n-1}{2})!} & \text{if $n$ is odd} \end{array}\right..\end{equation}
\end{definition}
From the definition, we can derive the factorial of an integer minus half as 
\begin{equation}\label{n-half}
\left(n-\frac{1}{2}\right)!=\frac{n!(2n-1)!!\sqrt{\pi}}{(2n)!!}.
\end{equation}

Moreover, the following identity can be used for the main results simplifications.
\begin{lemma} \label{combi-id} For an integer $n$, we have the following combinatorial identity 
\begin{equation}\label{main-id}\binom{n-\frac{1}{2}}{n-k}\binom{n-\frac{1}{2}}{k}=\frac{1}{2^{2n}}\binom{2n}{n}\binom{2n}{2k}.\end{equation}
\end{lemma}
\begin{proof}
By expanding the left-hand side and using \eqref{n-half} with some simplifications, we have
\begin{align*}
\binom{n-\frac{1}{2}}{n-k}\binom{n-\frac{1}{2}}{k}&=\frac{\left(n-\frac{1}{2}\right)!}{(n-k)!(k-\frac{1}{2})!}\frac{\left(n-\frac{1}{2}\right)!}{k!(n-k-\frac{1}{2})!}\\
&=\frac{(2n-1)!!}{2^{n}(n-k)!}\frac{2^{k}}{(2k-1)!!}\frac{(2n-1)!!}{2^{n}k!}\frac{2^{n-k}}{(2n-2k-1)!!}\\
&=\frac{1}{2^{n}(n-k)!k!}\frac{(2n-1)!!}{(2k-1)!!}\frac{(2n-1)!!}{(2(n-k)-1)!!}.
\end{align*}
Using the fact $(2n)!=(2n-1)!!2^{n}n!$ we get the desired identity.
\end{proof}

\subsection{Univariate Chebyshev-I Polynomials}\label{Ultrasp-orthog}
The Chebyshev-I polynomials $T_{n}(x)$ of degree $n$ are the orthogonal, except for a constant factor, on $[-1,1]$ with respect to the weight function 
$\mathrm{W}(x)=\frac{1}{\sqrt{1-x^{2}}}.$

The following two lemmas will be used in the construction of the bivariate Chebyshev-I weighted orthogonal polynomials and the proof of the main results.
The Pochhammer symbol is more appropriate, but the combinatorial notation gives
more compact and readable formulas, these have also been used by Szeg\"{o} \cite{Szego}. 
\begin{lemma}\label{jacinBer form}\cite{rababah4}
The Chebyshev-I polynomials $T_{r}(x)$ have the
Bernstein representation:
\begin{equation}
T_{r}(x)=
\frac{r!(2r-1)!!}{2^{2r}(2r)!!}\binom{2r}{r}^{2}
\sum_{i=0}^{r}(-1)^{r-i}\frac{i!(r-i)!}{(2i)!(2r-2i)!}B_{i}^{r}(x), \hspace{.05in} r=0,1,\dots.
\end{equation}
\end{lemma}

\begin{lemma}\label{rabab2}\cite{rababah4}
The Chebyshev-I polynomials $T_{0}(x),\dots,T_{n}(x)$\ of degree less than or equal to $n$ can be expressed in the Bernstein basis of fixed degree $n$ by the
following formula
\[T_{r}(x)=\sum\limits_{i=0}^{n}M_{i,r}^{n}B_{i}^{n}(x) ,\quad
r=0,1,\ldots,n\]
where
\begin{equation}\label{coeff}
M_{i,r}^{n}=\binom{n}{i}^{-1}\sum\limits_{k=\max(0,i+r-n)}^{\min(i,r)}(-1)^{r-k}\binom{n-r}{i-k}\binom{2r}{2k}.
\end{equation}
\end{lemma}
For simplicity and without lost of generality,  we take $x\in[0,1]$ for both Bernstein and Chebyshev-I polynomials. 

\section{Chebyshev-I Weighted Orthogonal Polynomials}
In this section, a simple closed-form representation of degree-ordered system
of orthogonal polynomials $\mathscr{T}_{n,r}(u,v,w)$ is constructed on a triangular domain $T,$ by generalization of the construction in \cite{Farouki}.  These polynomials will be given in Bernstein form, since Bernstein polynomials are stable \cite{Farouki2}.
The basic idea in this construction is to make $\mathscr{T}_{n,r}(u,v,w)$ coincide with the univeriate
Chebyshev-I polynomial along one edge of $T,$ and to make its variation along each chord parallel to that edge a scaled version of this Chebyshev polynomial.
The variation of $\mathscr{T}_{n,r}(u,v,w)$ with $w$ can then be arranged so as to ensure its orthogonality on $T$ with
every polynomial of degree $<n,$ and with other basis polynomials $\mathscr{T}_{n,s}(u,v,w)$ of degree $n$ for $r\neq s.$

\begin{lemma}\label{faro}\cite{Farouki}.
For $r = 0,\dots,n$ define the polynomials
\begin{equation}\label{poly} Q_{n,r}(w)=\sum_{j=0}^{n-r} (-1)^{j}\binom{n+r+1}{j}B_{j}^{n-r}(w),\end{equation}
then for $i = 0,\dots,n-r-1,$ $Q_{n,r}(w)$ is orthogonal to $(1-w)^{2r+i+1}$ on $[0,1],$ and hence
$$\int_{0}^{1} Q_{n,r}(w)p(w)(1-w)^{2r+1}dw = 0$$
for every polynomial $p(w)$ of degree less than or equal $n-r-1.$
\end{lemma}

Now, for $r=0,1,\ldots ,n$ and $n=0,1,2,\ldots$ we define the bivariate polynomials  
\begin{equation}\label{jacBiv}
\mathscr{T}_{n,r}(u,v,w)=\sum_{i=0}^{r}c(i)B_{i}^{r}(u,v)\sum_{j=0}^{n-r}(-1)^{j}\binom{n+r+1}{j}B_{j}^{n-r}(w,u+v),
\end{equation}
where $B_{i}^{r}(u,v)$ is the Bernstein polynomials defined in \eqref{bern} and
\begin{equation}
c(i)=(-1)^{r-i}\frac{\binom{2r}{r}\binom{2r}{2i}}{2^{2r}\binom{r}{i}},\hspace{.1in}i=0,\dots,r.
\end{equation}

To show that the bivariate polynomials $\mathscr{T}_{n,r}(u,v,w),$ $r=0,1,\ldots,n;$
$n\geq 0$ with respect to the weight function $\mathrm{W}(u,v,w),$  form an orthogonal system over the triangular domain $T,$ 
we prove $\mathscr{T}_{n,r}(u,v,w)\in \mathfrak{L}_{n},$  $r=0,1,\ldots,n;$ $n\geq 1,$ and for $r\neq s,$
$\mathscr{T}_{n,r}\perp \mathscr{T}_{n,s}.$ 

Let $\mathscr{T}_{0,0}=1,$ the polynomials
$\mathscr{T}_{n,r}(u,v,w)$ for $r=0,\dots,n;$ $n\geq 0$ form a degree-ordered orthogonal sequence over T. 

The polynomials in \eqref{jacBiv} can be written using the univariate Chebyshev-I polynomials form as:
\begin{equation}\mathscr{T}_{n,r}(u,v,w)=\sum\limits_{i=0}^{r}c(i)\frac{B_{i}^{r}(u,v)}{(u+v)^{r}}(1-w)^{r}\sum\limits_{j=0}^{n-r}(-1)^{j}\binom{n+r+1}{j} B_{j}^{n-r}(w,1-w).
\end{equation}
Using Lemma \ref{jacinBer form} and $B_{i}^{r}(u,v)=(u+v)^{r}B_{i}^{r}(\frac{u}{1-w}),$ we get
\begin{equation}\label{univ-poly}
\mathscr{T}_{n,r}(u,v,w)=\frac{2^{2r}(r!)^{2}}{(2r)!}T_{r}(\frac{{u}}{{1-w}})(1-w)^{r}Q_{n,r}(w),\hspace{.1in}r=0,\dots,n,
\end{equation}
where $T_{r}(t)$\ is the univariate Chebyshev-I polynomial of degree $r$ and $Q_{n,r}(w)$ is defined in equation \eqref{poly}.
For simplicity and without loss of generality we rewrite \eqref{univ-poly} as
\begin{equation}
\mathscr{T}_{n,r}(u,v,w)=T_{r}(\frac{{u}}{{1-w}})(1-w)^{r}Q_{n,r}(w),\hspace{.1in}r=0,\dots,n.
\end{equation}

Now, we show that the polynomials
$\mathscr{T}_{n,r}(u,v,w),$ $r=0,\dots,n,$
are orthogonal to all polynomials of degree $<n$ over the
triangular domain $T$.

For each $s=0,\dots,m$ and $m=0,\dots,n-1$ we define the bivariate polynomials
\begin{equation}
g_{s,m}(u,v,w)=T_{s}(\frac{u}{1-w})(1-w)^{m} w^{n-m-1}, \hspace{.05in}m=0,\dots,n-1,s=0,\dots,m.
\end{equation}
The span of $g_{s,m}(u,v,w)$ includes the set of Bernstein polynomials
$B_{j}^{m}(u,v)w^{n-m-1}, j=0,\dots,m; m=0,\dots,n-1,$
which span $\Pi_{n-1}.$ 

So, it is sufficient to show that for each $s=0,\dots,m;$ $m=0,\dots,n-1,$
\begin{equation}I:=
\iint\limits_{T}\mathscr{T}_{n,r}(u,v,w)g_{s,m}(u,v,w)\mathrm{W}(u,v,w)dA=0.
\end{equation}
The integral $I$ can be simplified as
\begin{equation}
I=\Delta \int\limits_{0}^{1}\int\limits_{0}^{1-w}T_{r}(%
\frac{u}{1-w})Q_{n,r}(w)T_{s}(\frac{u}{1-w}%
)w^{n-m-1}u^{-\frac{1}{2}}v^{-\frac{1}{2}}(1-w)^{\gamma+r+m }dudw.
\end{equation}
By making the substitution $u=t(1-w),$ we have
\begin{equation*}
I=\Delta \int\limits_{0}^{1}T_{r}(t)T_{s}(t)t^{-\frac{1}{2}}\left(1-t\right)
^{-\frac{1}{2}}dt\int\limits_{0}^{1}Q_{n,r}(w)(1-w)^{\gamma+r+m}w^{n-m-1}dw.
\end{equation*}
If $m<r$\ then we have $s<r,$\ and the first integral is zero by
the orthogonality property of the Chebyshev-I polynomials. If
$r\leq m\leq n-1$, the second integral equals zero by Lemma \ref{faro}. Thus we have the following theorem.

\begin{theorem}\label{main}
For each $r=0,1,\ldots,n;$ $n\geq 1,$  $\mathscr{T}_{n,r}(u,v,w)\in \mathfrak{L}_{n}$ with respect to the weight function
$\mathrm{W}(u,v,w)=u^{-\frac{1}{2}}v^{-\frac{1}{2}}(1-w)^{\gamma}$ such that $\gamma\geq 1.$
\end{theorem}
Note that taking $\mathrm{W}(u,v,w)=u^{-\frac{1}{2}}v^{-\frac{1}{2}}(1-w)^{\gamma}$ enables us to separate the integrand, and taking the constrain
$\gamma\geq 1$ enables us to use Lemma \ref{faro}.

Now we need to show $\mathscr{T}_{n,r}(u,v,w)$ is orthogonal
to each polynomial $\mathscr{T}_{n,s}(u,v,w)$ where $r\neq s.$ 
\begin{align*}I&=
\iint\nolimits_{T}\mathscr{T}_{n,r}(u,v,w)\mathscr{T}_{n,s}(u,v,w)\mathrm{W}(u,v,w)dA\\
&=\Delta \int\limits_{0}^{1}\int\limits_{0}^{1-w}T_{r}\left(\frac{u}{1-w}\right) T_{s}\left(\frac{u}{1-w}%
\right) (1-w)^{r+s}Q_{n,r}(w)Q_{n,s}(w)\mathrm{W}(u,v,w)dudw.
\end{align*}
By making the substitution $u=t(1-w),$ we have
\[I=\Delta \int\limits_{0}^{1}T_{r}(t)T_{s}(t)t^{-\frac{1}{2}}\left(1-t\right)^{-\frac{1}{2}}dt\int\limits_{0}^{1}Q_{n,r}(w)Q_{n,s}(w)(1-w)^{\gamma+r+s}dw.
\]
The previous integral equals zero by orthogonality property of the Chebyshev-I polynomials, thus we have the following theorem.

\begin{theorem}
For $r\neq s,$ $\mathscr{T}_{n,r}(u,v,w)\perp
\mathscr{T}_{n,s}(u,v,w)$ with respect to
the weight function
$\mathrm{W}(u,v,w)=u^{-\frac{1}{2}}v^{-\frac{1}{2}}(1-w)^{\gamma}$
such that $\gamma >-1$ $.$
\end{theorem}
Therefore, the bivariate
polynomials $\mathscr{T}_{n,r}(u,v,w),$ $r=0,1,\ldots,n;$ $n\geq 0$ form an orthogonal system over the triangular domain $T$ with respect to the weight function $\mathrm{W}(u,v,w),$  $\gamma>-1.$

\section{Applications}\label{Orthog-poly-Berns-basis}
The Bernstein-B\'{e}zier form of curves and surfaces have some interesting geometric properties \cite{Farin,Hoschek}, which is very important in the
numerical computations. So, the orthogonal polynomials $\mathscr{T}_{n,r}(u,v,w),$ $r=0,1,\dots,n;$ $n\geq 0$ can be written in the following
Bernstein-B\'{e}zier form:
\begin{equation}\label{bezier}
\mathscr{T}_{n,r}(u,v,w)=\sum\limits_{\left|\zeta\right| =n}a_{\zeta}^{n,r}B_{\zeta}^{n}(u,v,w).
\end{equation}

We are interested in finding a closed form of the Bernstein coefficients $a_{\zeta}^{n,r}$ and derive a recursion relation that allow us to compute the
coefficients efficiently.

From equation \eqref{jacBiv}, $\mathscr{T}_{n,r}(u,v,w)$ has degree
$\leq n-r$ in the variable $w,$ so
\begin{equation} a_{ijk}^{n,r}=0 \hspace{.1in} \text{for} \hspace{.1in} k>n-r.
\end{equation}
For $0\leq k\leq n-r,$ the remaining coefficients are determined
by equating \eqref{jacBiv} and \eqref{bezier} as follows
\[\sum\limits_{i+j=n-k}a_{ijk}^{n,r}B_{ijk}^{n}(u,v,w)=(-1)^{k}\binom{n+r+1}{k}B_{k}^{n-r}(w,u+v)\sum\limits_{i=0}^{r}c(i)B_{i}^{r}(u,v).
\]
Comparing powers of $w$ on both sides, we have
\[\sum\limits_{i=0}^{n-k}a_{ijk}^{n,r}\frac{n!}{i!j!k!}u^{i}v^{j}=(-1)^{k} \binom{n+r+1}{k} \binom{n-r}{k}(u+v)^{n-r-k}\sum\limits_{i=0}^{r}c(i)B_{i}^{r}(u,v).\]
The left hand side of the last equation can be written in the form
\begin{equation*}
\sum\limits_{i=0}^{n-k}a_{ijk}^{n,r}\frac{n!(n-k)!}{i!(n-k-i)!k!(n-k)!}u^{i}v^{j}
=\sum\limits_{i=0}^{n-k}a_{ijk}^{n,r}\binom{n}{k}B_{i}^{n-k}(u,v).
\end{equation*}
Now, we get
\[\sum\limits_{i=0}^{n-k}a_{ijk}^{n,r}\binom{n}{k}B_{i}^{n-k}(u,v)=(-1)^{k}\binom{n+r+1}{k}\binom{n-r}{k} (u+v)^{n-r-k}\sum\limits_{i=0}^{r}c(i)B_{i}^{r}(u,v).
\]
With some binomial simplifications and using Lemma \ref{rabab2}, we get
\begin{equation}
\sum\limits_{i=0}^{n-k}a_{ijk}^{n,r}\binom{n}{k}B_{i}^{n-k}(u,v)
=(-1)^{k}\binom{n+r+1}{k}\binom{n-r}{k}\sum\limits_{i=0}^{r}M_{i,r}^{n-k}B_{i}^{n-k}(u,v),\end{equation}
where $M_{i,r}^{n-k}$  are the coefficients resulting from
writing Chebyshev-I polynomial of degree $r$ in the Bernstein basis of
degree $n-k,$ as defined by expression \eqref{coeff}. Thus, the required
Bernstein-B\'{e}zier coefficients given by the following theorem.

\begin{theorem}\label{berncoeff}
The Bernstein coefficients $a_{\zeta}^{n,r}$ of equation \eqref{bezier} are given explicitly by
\begin{equation}
a_{ijk}^{n,r}=\left\{\begin{array}{ll} (-1)^{k}\frac{\binom{n+r+1}{k}\binom{n-r}{k}}{\binom{n}{k}}
M_{i,r}^{n-k} & 0\leq k\leq n-r \\
0 & k > n-r
\end{array} \right.,
\end{equation}
where $M_{i,r}^{n-k}$ are given in \eqref{coeff}.
\end{theorem}

To derive a recurrence relation for the coefficients $a_{ijk}^{n,r}$ of $\mathscr{T}_{n,r}(u,v,w),$ 
consider the generalized Bernstein polynomial of degree $n-1,$
\begin{equation*}
B_{ijk}^{n-1}(u,v,w)=\frac{(i+1)}{n}B_{i+1,j,k}^{n}(u,v,w)+\frac{(j+1)}{n}B_{i,j+1,k}^{n}(u,v,w)+%
\frac{(k+1)}{n}B_{i,j,k+1}^{n}(u,v,w).
\end{equation*}
From the construction of $\mathscr{T}_{n,r}(u,v,w),$ we have $\langle
B_{ijk}^{n-1}(u,v,w),\mathscr{T}_{n,r}(u,v,w)\rangle=0,$ where $i+j+k=n-1.$
Using Lemma \ref{iffthm},
\begin{equation}\label{aa}
(i+1)a_{i+1,j,k}^{n,r}+(j+1)a_{i,j+1,k}^{n,r}+(k+1)a_{i,j,k+1}^{n,r}=0.
\end{equation}
By Theorem \ref{berncoeff}, we have
$a_{i,n-i,0}^{n,r}=M_{i,r}^{n}\mbox{ for } i=0,1,\ldots ,n.$
Thus, we can use \eqref{aa} to generate $a_{i,j,k}^{n,r}$ recursively on $k.$

\subsection*{Acknowledgments}
 The author would like to thank the anonymous referee for his/her comments that helped to improve this article.

\end{document}